\documentclass[11pt,a4paper]{amsart}
\usepackage{mathrsfs}
\usepackage{syntonly}
\usepackage{amsmath}
\usepackage{amsthm}
\usepackage{amsfonts}
\usepackage{amssymb}
\usepackage{latexsym}
\usepackage{amscd,amssymb,amsopn,amsmath,amsthm,graphics,amsfonts,mathrsfs,accents,enumerate,verbatim,calc}
\usepackage[dvips]{graphicx}
\usepackage[colorlinks=true,linkcolor=red,citecolor=blue]{hyperref}
\usepackage[all]{xy}
\usepackage{enumitem}

\date{}
\allowdisplaybreaks[4] \footskip=15pt
\renewcommand{\uppercasenonmath}[1]{}

\numberwithin{equation}{section} \theoremstyle{plain}
\newtheorem{lem}{Lemma}[section]
\newtheorem{cor}[lem]{Corollary}
\newtheorem{prop}[lem]{Proposition}
\newtheorem{thm}[lem]{Theorem}

\newtheorem{cond}[lem]{Condition}
\newtheorem{definition}[lem]{Definition}
\newtheorem{Ex}[lem]{Example}
\newtheorem{Quest}[lem]{Question}
\newtheorem{Property}[lem]{Property}
\newtheorem{Properties}[lem]{Properties}
\newtheorem{Subprops}{}[lem]
\newtheorem{Para}[lem]{}

\newtheorem{remark}[lem]{Remark}
\newtheorem{rem}[lem]{Remark}

\newtheorem*{ack*}{ACKNOWLEDGEMENTS}

\newcommand{\pf}{\noindent\begin {proof}}
\newcommand{\epf}{\end{proof}}

\newcommand{\ra}{\rightarrow}

\begin{document}
\begin{center}
{\Large  \bf Balance of complete cohomology in extriangulated categories}

\vspace{0.5cm}  Jiangsheng Hu, Dongdong Zhang, Tiwei Zhao\footnote{Corresponding author. \\ Jiangsheng Hu was supported by the NSF of China (Grants Nos. 11671069, 11771212), Qing Lan Project of Jiangsu Province and Jiangsu Government Scholarship for Overseas Studies (JS-2019-328). Tiwei Zhao was supported by the NSF of China (Grants Nos. 11971225, 11901341), the project ZR2019QA015 supported by Shandong Provincial Natural Science Foundation, and the Young Talents Invitation Program of Shandong Province. Panyue Zhou was supported by the National Natural Science Foundation of China (Grant Nos. 11901190, 11671221),  the Hunan Provincial Natural Science Foundation of China (Grant No. 2018JJ3205) and  the Scientific Research Fund of Hunan Provincial Education Department (Grant No. 19B239).} and Panyue Zhou
\end{center}
\medskip
\medskip
\bigskip
\centerline { \bf  Abstract}
\medskip
\leftskip10truemm \rightskip10truemm \noindent
\hspace{1em}Let $(\mathcal{C},\mathbb{E},\mathfrak{s})$ be an extriangulated category with a proper class $\xi$ of
$\mathbb{E}$-triangles.  In this paper, we study the balance of complete cohomology in $(\mathcal{C},\mathbb{E},\mathfrak{s})$, which is motivated by
a result of Nucinkis that complete cohomology of modules is not balanced in the way the absolute cohomology Ext is balanced. As an application, we give some criteria for identifying a triangulated catgory to be Gorenstein and an artin algebra to be $F$-Gorenstein.
\\[2mm]

\leftskip0truemm \rightskip0truemm
\section { \bf Introduction}

Exact categories and triangulated categories are two fundamental structures in different branches of mathematics.
 As expected, exact categories and triangulated categories are not
independent of each other. In \cite{NP}, Nakaoka and Palu introduced the notion of externally triangulated categories
(extriangulated categories for short) as a simultaneous generalization of exact categories and extension-closed subcategories of triangulated
categories (they may no longer be triangulated
categories in general). After that, the study of extriangulated categories has become an active topic, and   up to now, many results on exact categories
and triangulated categories can be unified in the same framework, e.g.
see \cite{HZZ, LN, NP,ZZ}.  Recently, the authors \cite{HZZ} studied a relative homological algebra in an extriangulated category $(\mathcal{C},\mathbb{E},\mathfrak{s})$ which parallels the relative homological algebra in a triangulated category. By specifying a class of $\mathbb{E}$-triangles, which is called a proper class $\xi$ of $\mathbb{E}$-triangles, the authors introduced $\xi$-projective, $\xi$-injective, $\xi$-$\mathcal{G}$projective and  $\xi$-$\mathcal{G}$injective dimensions, and discussed their properties.

It is well known that for any modules $M$ and $N$ over a ring $R$, a projective resolution of
$M$ and an injective coresolution of $N$ lead to the same cohomolgy group $\textrm{Ext}^{\ast}_{R}(M,N)$. This implies that the absolute cohomology Ext is balanced, which is important and fundamental for classical homological algebra. Mislin \cite{Mislin} and Nucinkis \cite{Nucinkis} defined complete cohomology of modules.   However, the complete cohomology of modules is not balanced in the way Ext is balanced by \cite[Theorem 5.2]{Nucinkis}.  Recently, the authors \cite{HZZZ}  developed a complete cohomology theory in an extriangulated category and demonstrated that this theory shared some basic properties of complete cohomology in the category of modules \cite{BF,Goichot,Mislin,Nucinkis,YC} and Tate cohomology in the triangulated category \cite{AS2,RL,RL3}.
It seems natural to characterize when complete cohomology in extriangulated categories is balanced. The aim of this paper is to study this question.

We now outline the results of the paper. In Section 2, we summarize some preliminaries
and basic facts about extriangulated categories which will be used throughout the paper.

From Section 3, we assume that $(\mathcal{C}, \mathbb{E}, \mathfrak{s})$ is an extriangulated category with enough $\xi$-projectives and enough $\xi$-injectives satisfying Condition (WIC). We first recall some definitions and basic properties of $\xi$-complete cohomology groups in $(\mathcal{C}, \mathbb{E}, \mathfrak{s})$, and then we prove that there are two long exact sequences of $\xi$-complete cohomology under some certain conditions (see Theorems \ref{thm3.6}, \ref{thm:3.11'} and \ref{thm3.11}).

In Section 4, we first show that for all objects $M$ and $N$ in $(\mathcal{C}, \mathbb{E}, \mathfrak{s})$, if $M$ has finite $\xi$-$\mathcal{G}$projective dimension and $N$ has finite $\xi$-$\mathcal{G}$injective dimension, then $\xi$-complete cohomology groups $\widetilde{\xi{\rm xt}}^i_{\mathcal{P}}(M,N)$ and $\widetilde{\xi{\rm xt}}^i_{\mathcal{I}}(M,N)$  are isomorphic for any $i\in{\mathbb{Z}}$ (see Proposition \ref{prop:key-result}), which improves \cite[Theorem 4.11]{AS2}, \cite[Theorem 2]{Ia} and \cite[Main Theorem]{RL}. As a result, we characterize when $\xi$-complete cohomology in $(\mathcal{C}, \mathbb{E}, \mathfrak{s})$ is balanced (see Theorem \ref{thm:3.4}). As consequences,  some criteria for a triangulated catgory to be Gorenstein and an artin algebra to be $F$-Gorenstein are given (see Corollaries \ref{cor:3.7} and \ref{cor3}).

\section{\bf Preliminaries}
Throughout this paper, we always assume that  $\mathcal{C}=(\mathcal{C}, \mathbb{E}, \mathfrak{s})$ is an extriangulated category and $\xi$ is a proper class of $\mathbb{E}$-triangles in  $\mathcal{C}$.  We also assume that the extriangulated category $\mathcal{C}$ has enough $\xi$-projectives and enough $\xi$-injectives satisfying Condition (WIC). Next we briefly recall some definitions and basic properties of extriangulated categories from \cite{NP}.
We omit some details here, but the reader can find them in \cite{NP}.

Let $\mathcal{C}$ be an additive category equipped with an additive bifunctor
$$\mathbb{E}: \mathcal{C}^{\rm op}\times \mathcal{C}\rightarrow {\rm Ab},$$
where ${\rm Ab}$ is the category of abelian groups. For any objects $A, C\in\mathcal{C}$, an element $\delta\in \mathbb{E}(C,A)$ is called an $\mathbb{E}$-extension.
Let $\mathfrak{s}$ be a correspondence which associates an equivalence class $$\mathfrak{s}(\delta)=\xymatrix@C=0.8cm{[A\ar[r]^x
 &B\ar[r]^y&C]}$$ to any $\mathbb{E}$-extension $\delta\in\mathbb{E}(C, A)$. This $\mathfrak{s}$ is called a {\it realization} of $\mathbb{E}$, if it makes the diagrams in \cite[Definition 2.9]{NP} commutative.
 A triplet $(\mathcal{C}, \mathbb{E}, \mathfrak{s})$ is called an {\it extriangulated category} if it satisfies the following conditions.
\begin{enumerate}
\item $\mathbb{E}\colon\mathcal{C}^{\rm op}\times \mathcal{C}\rightarrow \rm{Ab}$ is an additive bifunctor.

\item $\mathfrak{s}$ is an additive realization of $\mathbb{E}$.

\item $\mathbb{E}$ and $\mathfrak{s}$  satisfy the compatibility conditions in \cite[Definition 2.12]{NP}.

 \end{enumerate}

\begin{rem}
Note that both exact categories and triangulated categories are extriangulated categories $($see \cite[Example 2.13]{NP}$)$ and extension closed subcategories of extriangulated categories are
again extriangulated $($see \cite[Remark 2.18]{NP}$)$. However, there exist lots of extriangulated categories which
are neither exact categories nor triangulated categories $($see \cite[Proposition 3.30]{NP}, \cite[Example 4.14]{ZZ} and \cite[Remark 3.3]{HZZ}$)$.
\end{rem}

We will use the following terminology.

\begin{definition}{ \emph{(see \cite[Definitions 2.15 and 2.19]{NP})}} {\rm
 Let $(\mathcal{C}, \mathbb{E}, \mathfrak{s})$ be an extriangulated category.
\begin{enumerate}
\item A sequence $\xymatrix@C=1cm{A\ar[r]^x&B\ar[r]^{y}&C}$ is called a {\it conflation} if it realizes some $\mathbb{E}$-extension $\delta\in\mathbb{E}(C, A)$.
In this case, $x$ is called an {\it inflation} and $y$ is called a {\it deflation}.

\item  If a conflation $\xymatrix@C=0.6cm{A\ar[r]^x&B\ar[r]^{y}&C}$ realizes $\delta\in\mathbb{E}(C, A)$, we call the pair
$\xymatrix@C=0.6cm{(A\ar[r]^x&B\ar[r]^{y}&C, \delta)}$ an {\it $\mathbb{E}$-triangle}, and write it in the following.
\begin{center} $\xymatrix{A\ar[r]^x&B\ar[r]^{y}&C\ar@{-->}[r]^{\delta}&}$\end{center}
We usually do not write this ``$\delta$" if it is not used in the argument.

\item Let $\xymatrix{A\ar[r]^x&B\ar[r]^{y}&C\ar@{-->}[r]^{\delta}&}$ and $\xymatrix{A'\ar[r]^{x'}&B'\ar[r]^{y'}&C'\ar@{-->}[r]^{\delta'}&}$
be any pair of $\mathbb{E}$-triangles. If a triplet $(a, b, c)$ realizes $(a, c): \delta\rightarrow \delta'$, then we write it as
 $$\xymatrix{A\ar[r]^{x}\ar[d]_{a}&B\ar[r]^{y}\ar[d]_{b}&C\ar[d]_{c}\ar@{-->}[r]^{\delta}&\\
 A'\ar[r]^{x'}&B'\ar[r]^{y'}&C'\ar@{-->}[r]^{\delta'}&}$$
 and call $(a, b, c)$ a {\it morphism} of $\mathbb{E}$-triangles.
\end{enumerate}}

\end{definition}

The following condition is analogous to the weak idempotent completeness in exact category (see \cite[Condition 5.8]{NP}).

\begin{cond} \label{cond:4.11} \emph{({\rm Condition (WIC)})}  Consider the following conditions.

\begin{enumerate}
\item  Let $f\in\mathcal{C}(A, B), g\in\mathcal{C}(B, C)$ be any composable pair of morphisms. If $gf$ is an inflation, then so is $f$.

\item Let $f\in\mathcal{C}(A, B), g\in\mathcal{C}(B, C)$ be any composable pair of morphisms. If $gf$ is a deflation, then so is $g$.

\end{enumerate}

\end{cond}

\begin{Ex}\label{Ex:4.12}

\emph{(1)} If $\mathcal{C}$ is an exact category, then Condition \emph{(WIC)} is equivalent to $\mathcal{C}$ is
weakly idempotent complete \emph{(see \cite[Proposition 7.6]{B"u})}.

\emph{(2)} If $\mathcal{C}$ is a triangulated category, then Condition \emph{(WIC)} is automatically satisfied.
\end{Ex}

\begin{lem}\label{lem1} \emph{(see \cite[Proposition 3.15]{NP})} Assume that $(\mathcal{C}, \mathbb{E},\mathfrak{s})$ is an extriangulated category. Let $C$ be any object, and let $\xymatrix@C=2em{A_1\ar[r]^{x_1}&B_1\ar[r]^{y_1}&C\ar@{-->}[r]^{\delta_1}&}$ and $\xymatrix@C=2em{A_2\ar[r]^{x_2}&B_2\ar[r]^{y_2}&C\ar@{-->}[r]^{\delta_2}&}$ be any pair of $\mathbb{E}$-triangles. Then there is a commutative diagram
in $\mathcal{C}$
$$\xymatrix{
    & A_2\ar[d]_{m_2} \ar@{=}[r] & A_2 \ar[d]^{x_2} \\
  A_1 \ar@{=}[d] \ar[r]^{m_1} & M \ar[d]_{e_2} \ar[r]^{e_1} & B_2\ar[d]^{y_2} \\
  A_1 \ar[r]^{x_1} & B_1\ar[r]^{y_1} & C   }
  $$
  which satisfies $\mathfrak{s}(y^*_2\delta_1)=\xymatrix@C=2em{[A_1\ar[r]^{m_1}&M\ar[r]^{e_1}&B_2]}$ and
  $\mathfrak{s}(y^*_1\delta_2)=\xymatrix@C=2em{[A_2\ar[r]^{m_2}&M\ar[r]^{e_2}&B_1]}$.

\end{lem}

The following definitions are quoted verbatim from \cite[Section 3]{HZZ}. A class of $\mathbb{E}$-triangles $\xi$ is {\it closed under base change} if for any $\mathbb{E}$-triangle $$\xymatrix@C=2em{A\ar[r]^x&B\ar[r]^y&C\ar@{-->}[r]^{\delta}&\in\xi}$$ and any morphism $c\colon C' \to C$, then any $\mathbb{E}$-triangle  $\xymatrix@C=2em{A\ar[r]^{x'}&B'\ar[r]^{y'}&C'\ar@{-->}[r]^{c^*\delta}&}$ belongs to $\xi$.

Dually, a class of  $\mathbb{E}$-triangles $\xi$ is {\it closed under cobase change} if for any $\mathbb{E}$-triangle $$\xymatrix@C=2em{A\ar[r]^x&B\ar[r]^y&C\ar@{-->}[r]^{\delta}&\in\xi}$$ and any morphism $a\colon A \to A'$, then any $\mathbb{E}$-triangle  $\xymatrix@C=2em{A'\ar[r]^{x'}&B'\ar[r]^{y'}&C\ar@{-->}[r]^{a_*\delta}&}$ belongs to $\xi$.

A class of $\mathbb{E}$-triangles $\xi$ is called {\it saturated} if in the situation of Lemma \ref{lem1}, whenever  \\
$\xymatrix@C=2em{A_2\ar[r]^{x_2}&B_2\ar[r]^{y_2}&C\ar@{-->}[r]^{\delta_2 }&}$
 and $\xymatrix@C=2em{A_1\ar[r]^{m_1}&M\ar[r]^{e_1}&B_2\ar@{-->}[r]^{y_2^{\ast}\delta_1}&}$
 belong to $\xi$, then the  $\mathbb{E}$-triangle $$\xymatrix@C=2em{A_1\ar[r]^{x_1}&B_1\ar[r]^{y_1}&C\ar@{-->}[r]^{\delta_1 }&}$$  belongs to $\xi$.

An $\mathbb{E}$-triangle $\xymatrix@C=2em{A\ar[r]^x&B\ar[r]^y&C\ar@{-->}[r]^{\delta}&}$ is called {\it split} if $\delta=0$. It is easy to see that it is split if and only if $x$ is section or $y$ is retraction. The full subcategory  consisting of the split $\mathbb{E}$-triangles will be denoted by $\Delta_0$.

  \begin{definition} \emph{(see \cite[Definition 3.1]{HZZ})}\label{def:proper class} {\rm  Let $\xi$ be a class of $\mathbb{E}$-triangles which is closed under isomorphisms. Then $\xi$ is called a {\it proper class} of $\mathbb{E}$-triangles if the following conditions hold:

  \begin{enumerate}
\item  $\xi$ is closed under finite coproducts and $\Delta_0\subseteq \xi$.

\item $\xi$ is closed under base change and cobase change.

\item $\xi$ is saturated.

  \end{enumerate}}
  \end{definition}

 \begin{definition} \emph{(see \cite[Definition 4.1]{HZZ})}
 {\rm An object $P\in\mathcal{C}$  is called {\it $\xi$-projective}  if for any $\mathbb{E}$-triangle $$\xymatrix{A\ar[r]^x& B\ar[r]^y& C \ar@{-->}[r]^{\delta}& }$$ in $\xi$, the induced sequence of abelian groups $\xymatrix@C=0.6cm{0\ar[r]& \mathcal{C}(P,A)\ar[r]& \mathcal{C}(P,B)\ar[r]&\mathcal{C}(P,C)\ar[r]& 0}$ is exact. Dually, we have the definition of {\it $\xi$-injective} objects.}
\end{definition}

We denote by $\mathcal{P(\xi)}$ (resp. $\mathcal{I(\xi)}$) the class of $\xi$-projective (resp. $\xi$-injective) objects of $\mathcal{C}$. It follows from the definition that this subcategory $\mathcal{P}(\xi)$ and $\mathcal{I}(\xi)$ are full, additive, closed under isomorphisms and direct summands.

 An extriangulated  category $(\mathcal{C}, \mathbb{E}, \mathfrak{s})$ is said to  have {\it  enough
$\xi$-projectives} \ (resp. {\it  enough $\xi$-injectives}) provided that for each object $A$ there exists an $\mathbb{E}$-triangle $\xymatrix@C=2.1em{K\ar[r]& P\ar[r]&A\ar@{-->}[r]& }$ (resp. $\xymatrix@C=2em{A\ar[r]& I\ar[r]& K\ar@{-->}[r]&}$) in $\xi$ with $P\in\mathcal{P}(\xi)$ (resp. $I\in\mathcal{I}(\xi)$).

Let $\xymatrix@C=2.1em{K\ar[r]& P\ar[r]&A\ar@{-->}[r]& }$ be an $\mathbb{E}$-triangle in $\xi$ with $P\in\mathcal{P}(\xi)$, then we call $K$ the \emph{first $\xi$-syzygy} of $A$. An $n$th \emph{$\xi$-syzygy} of $A$ is defined as usual by induction. By Schanuel's lemma (\cite[Proposition 4.3]{HZZ}), any two $\xi$-syzygies
of $A$ are isomorphic modulo $\xi$-projectives.

The {\it $\xi$-projective dimension} $\xi$-${\rm pd} A$ of $A\in\mathcal{C}$ is defined inductively.
 If $A\in\mathcal{P}(\xi)$, then define $\xi$-${\rm pd} A=0$.
Next if $\xi$-${\rm pd} A>0$, define $\xi$-${\rm pd} A\leq n$ if there exists an $\mathbb{E}$-triangle
 $K\to P\to A\dashrightarrow$  in $\xi$ with $P\in \mathcal{P}(\xi)$ and $\xi$-${\rm pd} K\leq n-1$.
Finally we define $\xi$-${\rm pd} A=n$ if $\xi$-${\rm pd} A\leq n$ and $\xi$-${\rm pd} A\nleq n-1$. Of course we set $\xi$-${\rm pd} A=\infty$, if $\xi$-${\rm pd} A\neq n$ for all $n\geq 0$.

Dually we can define the {\it $\xi$-injective dimension}  $\xi$-${\rm id} A$ of an object $A\in\mathcal{C}$.


\begin{definition} \emph{(see \cite[Definition 4.4]{HZZ})}
{\rm A {\it $\xi$-exact} complex $\mathbf{X}$ is a diagram $$\xymatrix@C=2em{\cdots\ar[r]&X_1\ar[r]^{d_1}&X_0\ar[r]^{d_0}&X_{-1}\ar[r]&\cdots}$$ in $\mathcal{C}$ such that for each integer $n$, there exists an $\mathbb{E}$-triangle $\xymatrix@C=2em{K_{n+1}\ar[r]^{g_n}&X_n\ar[r]^{f_n}&K_n\ar@{-->}[r]^{\delta_n}&}$ in $\xi$ and $d_n=g_{n-1}f_n$.
}\end{definition}

\begin{definition} \emph{(see \cite[Definition 4.5]{HZZ})}
{\rm Let $\mathcal{W}$ be a class of objects in $\mathcal{C}$. An $\mathbb{E}$-triangle
$$\xymatrix@C=2em{A\ar[r]& B\ar[r]& C\ar@{-->}[r]& }$$ in $\xi$ is called to be
{\it $\mathcal{C}(-,\mathcal{W})$-exact} (resp.
{\it $\mathcal{C}(\mathcal{W},-)$-exact}) if for any $W\in\mathcal{W}$, the induced sequence of abelian groups $\xymatrix@C=2em{0\ar[r]&\mathcal{C}(C,W)\ar[r]&\mathcal{C}(B,W)\ar[r]&\mathcal{C}(A,W)\ar[r]& 0}$ (resp. \\ $\xymatrix@C=2em{0\ar[r]&\mathcal{C}(W,A)\ar[r]&\mathcal{C}(W,B)\ar[r]&\mathcal{C}(W,C)\ar[r]&0}$) is exact in ${\rm Ab}$}.
\end{definition}

\begin{definition} \emph{(see \cite[Definition 4.6]{HZZ})}
 {\rm Let $\mathcal{W}$ be a class of objects in $\mathcal{C}$. A complex $\mathbf{X}$ is called {\it $\mathcal{C}(-,\mathcal{W})$-exact} (resp.
{\it $\mathcal{C}(\mathcal{W},-)$-exact}) if it is a $\xi$-exact complex
$$\xymatrix@C=2em{\cdots\ar[r]&X_1\ar[r]^{d_1}&X_0\ar[r]^{d_0}&X_{-1}\ar[r]&\cdots}$$ in $\mathcal{C}$ such that  there is a $\mathcal{C}(-,\mathcal{W})$-exact (resp.
 $\mathcal{C}(\mathcal{W},-)$-exact) $\mathbb{E}$-triangle $$\xymatrix@C=2em{K_{n+1}\ar[r]^{g_n}&X_n\ar[r]^{f_n}&K_n\ar@{-->}[r]^{\delta_n}&}$$ in $\xi$ for each integer $n$ and $d_n=g_{n-1}f_n$.

 A $\xi$-exact complex $\mathbf{X}$ is called {\it complete $\mathcal{P}(\xi)$-exact} (resp. {\it complete $\mathcal{I}(\xi)$-exact}) if it is $\mathcal{C}(-,\mathcal{P}(\xi))$-exact (resp.
 $\mathcal{C}(\mathcal{I}(\xi),-)$-exact).}
\end{definition}

\begin{definition} \emph{(see \cite[Definition 4.7]{HZZ})}
 {\rm A  {\it complete $\xi$-projective resolution}  is a complete $\mathcal{P}(\xi)$-exact complex\\ $$\xymatrix@C=2em{\mathbf{P}:\cdots\ar[r]&P_1\ar[r]^{d_1}&P_0\ar[r]^{d_0}&P_{-1}\ar[r]&\cdots}$$ in $\mathcal{C}$ such that $P_n$ is $\xi$-projective for each integer $n$. Dually,   a  {\it complete $\xi$-injective coresolution}  is a complete $\mathcal{I}(\xi)$-exact complex $$\xymatrix@C=2em{\mathbf{I}:\cdots\ar[r]&I_1\ar[r]^{d_1}&I_0\ar[r]^{d_0}&I_{-1}\ar[r]&\cdots}$$ in $\mathcal{C}$ such that $I_n$ is $\xi$-injective for each integer $n$.}
\end{definition}

\begin{definition} \emph{(see \cite[Definition 4.8]{HZZ})}
{\rm  Let $\mathbf{P}$ be a complete $\xi$-projective resolution in $\mathcal{C}$. So for each integer $n$, there exists a $\mathcal{C}(-, \mathcal{P}(\xi))$-exact $\mathbb{E}$-triangle $\xymatrix@C=2em{K_{n+1}\ar[r]^{g_n}&P_n\ar[r]^{f_n}&K_n\ar@{-->}[r]^{\delta_n}&}$ in $\xi$. The objects $K_n$ are called {\it $\xi$-$\mathcal{G}$projective} for each integer $n$. Dually if  $\mathbf{I}$ is a complete $\xi$-injective  coresolution in $\mathcal{C}$, there exists a  $\mathcal{C}(\mathcal{I}(\xi), -)$-exact $\mathbb{E}$-triangle $\xymatrix@C=2em{K_{n+1}\ar[r]^{g_n}&I_n\ar[r]^{f_n}&K_n\ar@{-->}[r]^{\delta_n}&}$ in $\xi$ for each integer $n$. The objects $K_n$ are called {\it $\xi$-$\mathcal{G}$injective} for each integer $n$.}
\end{definition}


 We denote by $\mathcal{GP}(\xi)$ (resp. $\mathcal{GI}(\xi)$) the class of $\xi$-$\mathcal{G}$projective (resp. $\xi$-$\mathcal{G}$injective) objects.
It is obvious that $\mathcal{P(\xi)}$ $\subseteq$ $\mathcal{GP}(\xi)$ and $\mathcal{I(\xi)}$ $\subseteq$ $\mathcal{GI}(\xi)$.

\begin{definition} \emph{(see \cite[Definition 3.1]{HZZ1})}\label{df:resolution} {\rm Let $M$ be an object in $\mathcal{C}$. A {\it $\xi$-projective resolution} of $M$ is a $\xi$-exact complex $\mathbf{P}\rightarrow M$ such that $\mathbf{P}_n\in{\mathcal{P}(\xi)}$ for all $n\geq0$. Dually, a {\it $\xi$-injective coresolution} of $M$ is a $\xi$-exact complex $ M\rightarrow \mathbf{I}$ such that $\mathbf{I}_n\in{\mathcal{I}(\xi)}$ for all $n\leq0$.}
\end{definition}

We denote by $\textrm{Ch}(\mathcal{C})$ the category of complexes in $\mathcal{C}$; the objects are complexes and morphisms are chain maps. We write the complexes homologically, so an object $\mathbf{X}$ of $\textrm{Ch}(\mathcal{C})$ is of the form
$$\xymatrix@C=2em{\mathbf{X}:=\cdots \ar[r]&X_{n+1}\ar[r]^{d_{n+1}^{\mathbf{X}}}&X_n\ar[r]^{d_n^{\mathbf{X}}}&X_{n-1}\ar[r]&\cdots}.$$
The \emph{$i$th shift} of $\mathbf{X}$ is the complex $\mathbf{X}[i]$ with $n$th component $\mathbf{X}_{n-i}$ and differential $d_n^{\mathbf{X}[i]}=(-1)^{i}d_{n-i}^{\mathbf{X}}$. Assume that $\mathbf{X}$ and $\mathbf{Y}$ are complexes in $\textrm{Ch}(\mathcal{C})$.
A homomorphism $\xymatrix@C=2em{\varphi:\mathbf{X}\ar[r]&\mathbf{Y}}$ of degree $n$ is a family $(\varphi_i)_{i\in\mathbb{Z}}$ of morphisms $\xymatrix@C=2em{\varphi_i:X_i\ar[r]& Y_{i+n}}$ for all $i\in\mathbb{Z}$. In this case, we set $|\varphi|=n$. All such homomorphisms form an abelian group, denoted by $\mathcal{C}(\mathbf{X},\mathbf{Y})_n$, which is identified with $\prod_{i\in \mathbb{Z}}{\rm \mathcal{C}}(X_i,Y_{i+n})$. We let $\mathcal{C}(\mathbf{X},\mathbf{Y})$ be the complex of abelian groups with $n$th component $\mathcal{C}(\mathbf{X},\mathbf{Y})_n$ and differential $d(\varphi_i)=d_{i+n}^{\mathbf{Y}}\varphi_i-(-1)^n\varphi_{i-1}d_i^{\mathbf{X}}$ for $\varphi=(\varphi_i)\in\mathcal{C}(\mathbf{X},\mathbf{Y})_n$.
We refer to \cite{AFH,CFH} for more details.

\begin{remark} \emph{(see \cite[Remark 3.3]{HZZZ})}\label{df:resolution}\label{3.3} Let $M$ and $N$ be objects in $\mathcal{C}$.
\begin{enumerate}
\item Note that there are two $\xi$-projective resolutions $\xymatrix@C=2em{\mathbf{P}_M\ar[r]& M}$ and $\xymatrix@C=2em{\mathbf{P}_N\ar[r]& N}$ of $M$ and $N$, respectively. A homomorphism $\beta\in \mathcal{C}(\mathbf{P}_M,\mathbf{P}_N)$ is {\it bounded above} if $\beta_i=0$ for all $i\gg 0$. The subset $\overline{\mathcal{C}}(\mathbf{P}_M,\mathbf{P}_N)$, consisting of all bounded above homomorphisms, is a subcomplex with components
\begin{center}$\overline{\mathcal{C}}(\mathbf{P}_M,\mathbf{P}_N)_n=\{(\varphi_i)\in \mathcal{C}(\mathbf{P}_M,\mathbf{P}_N)_n \ | \ \varphi_i=0$ for all $i\gg 0\}.$\end{center}
We set
$$\widetilde{\mathcal{C}}(\mathbf{P}_M,\mathbf{P}_N)={\mathcal{C}}(\mathbf{P}_M,\mathbf{P}_N)/\overline{\mathcal{C}}(\mathbf{P}_M,\mathbf{P}_N).$$

\item Note that there are two $\xi$-injective coresolutions $\xymatrix@C=2em{M\ar[r]&\mathbf{I}_M}$ and $\xymatrix{N\ar[r]&\mathbf{I}_N}$ of $M$ and $N$, respectively. A homomorphism $\beta\in {\mathcal{C}}(\mathbf{I}_M,\mathbf{I}_N)$ is {\it bounded below} if $\beta_i=0$ for all $i\ll 0$. The subset $\underline{\mathcal{C}}(\mathbf{I}_M,\mathbf{I}_N)$, consisting of all bounded below homomorphisms, is a subcomplex with components
\begin{center}$\underline{\mathcal{C}}(\mathbf{I}_M,\mathbf{I}_N)_n=\{(\varphi_i)\in {\mathcal{C}}(\mathbf{I}_M,\mathbf{I}_N)_n \ | \ \varphi_i=0$ for all $i\ll 0\}.$\end{center}
We set
$$\widetilde{\mathcal{C}}(\mathbf{I}_M,\mathbf{I}_N)={\mathcal{C}}(\mathbf{I}_M,\mathbf{I}_N)/\underline{\mathcal{C}}(\mathbf{I}_M,\mathbf{I}_N).$$
\end{enumerate}
\end{remark}

\begin{definition} \emph{(see \cite[Definition 3.2]{HZZ1})}\label{df:derived-functors} {\rm Let $M$ and $N$ be objects in $\mathcal{C}$.

\begin{enumerate}
\item[{\rm (1)}] If we choose a $\xi$-projective resolution $\xymatrix@C=2em{\mathbf{P}\ar[r]& M}$ of  $M$, then for any integer $n\geq 0$, the \emph{$\xi$-cohomology groups} $\xi{\rm xt}_{\mathcal{P}(\xi)}^n(M,N)$ are defined as
$$\xi{\rm xt}_{\mathcal{P}(\xi)}^n(M,N)=H^n({\mathcal{C}}(\mathbf{P},N)).$$

\item[{\rm (2)}] If we choose a
$\xi$-injective coresolution $\xymatrix@C=2em{N\ar[r]&\mathbf{I}}$ of  $N$, then for any integer $n\geq 0$, the \emph{$\xi$-cohomology groups} $\xi{\rm xt}_{\mathcal{I}(\xi)}^n(M,N)$ are defined as $$\xi{\rm xt}_{\mathcal{I}(\xi)}^n(M,N)=H^n({\mathcal{C}}(M, \mathbf{I})).$$
\end{enumerate}}
\end{definition}

\begin{rem}\label{fact:2.5'} { By \cite[Lemma 3.2]{HZZZ}, one can see that ${\rm \xi xt}_{\mathcal{P}(\xi)}^n(-,-)$ and ${\rm \xi xt}_{\mathcal{I}(\xi)}^n(-,-)$ are cohomological functors for any integer $n\geq 0$, independent of the choice of $\xi$-projective resolutions and $\xi$-injective coresolutions, respectively. In fact, with the modifications of the usual proof, one obtains the isomorphism $\xi{\rm xt}_{\mathcal{P}(\xi)}^n(M,N)\cong \xi{\rm xt}_{\mathcal{I}(\xi)}^n(M,N),$
which is denoted by $\xi{\rm xt}_{\xi}^n(M,N).$
}
\end{rem}

\section{\bf $\xi$-complete cohomology and its long exact sequences}
The goal of this section is to study long exact sequences of $\xi$-complete cohomology, which gives some
preparations for the proof of the main result in the next section. To this end, we first recall some definitions and basic properties of $\xi$-complete cohomology in $\mathcal{C}$.

\begin{definition}\label{df:3.6} {\rm (see \cite[Definition 3.4]{HZZZ})} {\rm Let $M$ and $N$ be objects in $\mathcal{C}$, and let $n$ be an integer.
\begin{enumerate}
\item Using $\xi$-projective resolutions, we define the $n$th \emph{$\xi$-complete cohomology group}, denoted by $\widetilde{\rm \xi xt}_{\mathcal{P}}^n(M,N)$, as
$$\widetilde{\rm \xi xt}_{\mathcal{P}}^n(M,N)=H^n(\widetilde{\mathcal{C}}(\mathbf{P}_M,\mathbf{P}_N)),$$
where $\widetilde{\mathcal{C}}(\mathbf{P}_M,\mathbf{P}_N)$ is the complex defined in Remark \ref{3.3}(1).

\item Using $\xi$-injective coresolutions, we define the $n$th \emph{$\xi$-complete cohomology group}, denoted by $\widetilde{\rm \xi xt}_{\mathcal{I}}^n(M,N)$, as
$$\widetilde{\rm \xi xt}_{\mathcal{I}}^n(M,N)=H^n(\widetilde{\mathcal{C}}(\mathbf{I}_M,\mathbf{I}_N)),$$
where $\widetilde{\mathcal{C}}(\mathbf{I}_M,\mathbf{I}_N)$ is the complex defined in Remark \ref{3.3}(2).
\end{enumerate}}
\end{definition}

\begin{definition}\label{df:3.2} {\rm (see \cite[Definition 4.3]{HZZZ})} {\rm Let $M\in\mathcal{C}$ be an object. A \emph{$\xi$-complete resolution} of $M$ is a diagram $$\xymatrix@C=2em{\mathbf{T}\ar[r]^{\nu}&\mathbf{P}\ar[r]^{\pi}&M}$$ of morphisms of complexes satisfying:
  (1)  $\pi:\mathbf{P}\ra M$ is a $\xi$-projective resolution of $M$;
  (2) $\mathbf{T}$ is a complete $\xi$-projective resolution;
  (3) $\nu:\mathbf{T}\ra \mathbf{P}$ is a morphism such that $\nu_{i}$ $=$ {\rm id$_{T_{i}}$} for all $i\gg 0$.
 }
\end{definition}

The following lemma is very key, which helps us to compute $\xi$-complete cohomology for objects having finite $\xi$-$\mathcal{G}$projective dimension using $\xi$-complete resolutions.

{\begin{lem}\label{lem:complete-cohomology} {\rm (see \cite[Theorem 4.6]{HZZZ})} Let $M$ and $N$ be objects in $\mathcal{C}$. If $M$ admits a $\xi$-complete resolution $\xymatrix@C=2em{\mathbf{T}\ar[r]^{\nu}&\mathbf{P}\ar[r]^{\pi}&M,}$ then for any integer $i$, there exists  an isomorphism
$$\widetilde{\xi{\rm xt}}_{\mathcal{P}}^i(M,N)\cong H^i(\mathcal{C}(\mathbf{T},N)).$$
\end{lem}}

\begin{rem}Note that in the module categories and triangulated categories, the cohomology groups $H^n(\mathcal{C}(\mathbf{T},N))$ are called Tate cohomology, see \cite{AS2,AM} for more details. Motivated by this, for all objects $M$ and $N$ in $\mathcal{C}$, if $M$ admits a $\xi$-complete resolution $\xymatrix@C=2em{\mathbf{T}\ar[r]^{\nu}&\mathbf{P}\ar[r]^{\pi}&M,}$ then the cohomology group $H^n(\mathcal{C}(\mathbf{T},N))$ can be defined as Tate cohomoloy group in $\mathcal{C}$ for any integer $n$, which is also denoted by $\widetilde{\xi{\rm xt}}_{\mathcal{P}}^n(M,N)$ by Lemma \ref{lem:complete-cohomology}.
\end{rem}

Now assume that $M$ admits a $\xi$-complete resolution $\xymatrix@C=2em{\mathbf{T}\ar[r]^{\nu}&\mathbf{P}\ar[r]^{\pi}&M.}$ For each $n\in\mathbb{Z}$,
we have a comparison morphism
$$
\widetilde{\varepsilon}^n_{\mathcal{P}}(M,N): {\xi{\rm xt}}_{\xi}^n(M,N) \to \widetilde{\xi{\rm xt}}_{\mathcal{P}}^n(M,N)
$$
given by
$$
H^n\mathcal{C}(\nu,N): H^n\mathcal{C}(\mathbf{P},N)\to H^n\mathcal{C}(\mathbf{T},N).
$$

We denote by $\widetilde{{\mathcal{GP}}}(\xi)$ (resp. $\widetilde{{\mathcal{GI}}}(\xi)$)  the full subcategory of $\mathcal{C}$ whose objects have finite $\xi$-$\mathcal{G}$projective (resp. $\xi$-$\mathcal{G}$injective)  dimension.


\begin{lem}\label{com-homotopy} {\rm (see \cite[Lemma 4.5]{HZZZ})}
Let $\xymatrix@C=2em{\mathbf{T}\ar[r]^{\nu}&\mathbf{P}\ar[r]^{\pi}&M}$ and $\xymatrix@C=2em{\mathbf{T}'\ar[r]^{\nu'}&\mathbf{P}'\ar[r]^{\pi'}&M'}$ be  $\xi$-complete resolutions of $M$ and $M'$, respectively. For each morphism $\mu: M\to M'$, there exists a morphism $\overline{\mu}$, unique up to homotopy, making the right-hand square of the following diagram
$$
\xymatrix{\mathbf{T}\ar[r]^{\nu}\ar[d]^{\widetilde{\mu}}&\mathbf{P}\ar[r]^{\pi}\ar[d]^{\overline{\mu}}&M\ar[d]^{\mu}\\
\mathbf{T}'\ar[r]^{\nu'}&\mathbf{P}'\ar[r]^{\pi'}&M'}
$$
commutative, and for each choice of $\overline{\mu}$, there exists a morphism $\widetilde{\mu}$, unique up to homotopy, making the left-hand square commute up to homotopy.  In particular, if \emph{$\mu=\textrm{id}_{M}$}, then $\widetilde {\mu}$ and $\overline{\mu}$ are homotopy equivalences.
\end{lem}

\begin{rem}
The assignment $(M,N)\mapsto \widetilde{\xi{\rm xt}}_{\mathcal{P}}^n(M,N)$ defines a functor
$$
\widetilde{\xi{\rm xt}}_{\mathcal{P}}^n: \widetilde{{\mathcal{GP}}}(\xi)^{\rm op}\times \mathcal{C}\to {\mathcal{A}}b
$$
 and the maps $\widetilde{\varepsilon}^n_{\mathcal{P}}(M,N)$ yields a morphism of functors $\widetilde{\varepsilon}^n_{\mathcal{P}}: {\xi{\rm xt}}_{\xi}^n\to \widetilde{\xi{\rm xt}}_{\mathcal{P}}^n$ such that both $\widetilde{\xi{\rm xt}}_{\mathcal{P}}^n$ and $\widetilde{\varepsilon}^n_{\mathcal{P}}$ are independent of the choice of resolutions and liftings.
\end{rem}

\begin{lem}\label{3.4}
Let
 $$\xymatrix@C=2em{\mathbf{Q}:\cdots\ar[r]&Q_1\ar[r]^{d_1}&Q_0\ar[r]^{d_0}&Q_{-1}\ar[r]&\cdots}$$
 be a complete $\xi$-projective resolution. If $M\in\widetilde{\mathcal{P}}(\xi)$ or $M\in\widetilde{\mathcal{I}}(\xi)$, then $\mathcal{C}(\mathbf{Q},M)$ is exact.
\end{lem}

\begin{proof}
It is easy to check by \cite[Lemma 5.3]{HZZ} and \cite[Lemma 4.5]{HZZ1}.
\end{proof}

\begin{prop}\label{prop3.5}
\begin{itemize}
\item[ ]
  \item [(1)] Let $M\in\widetilde{\mathcal{GP}}(\xi)$. For any integer $n$ the following are equivalent:
  \begin{itemize}
    \item [(i)] $\xi\mbox{-}\mathcal{G}{\rm pd}M\leq n$.
    \item [(ii)] The map $\widetilde{\varepsilon}^i_{\mathcal{P}}(M,N):\xi{\rm xt}^i_{\xi}(M,N)\to \widetilde{\xi{\rm xt}}^i_{\mathcal{P}}(M,N)$ is bijective for any $i\geq n+1$ and any $N\in\mathcal{C}$.
  \end{itemize}
  \item [(2)] If $\xi\mbox{-}{\rm pd}M<\infty$, then $\widetilde{\xi{\rm xt}}^i_{\mathcal{P}}(M,-)=0$ for any $i\in\mathbb{Z}$.
  \item [(3)] If $\xi\mbox{-}{\rm pd}M<\infty$, then $\widetilde{\xi{\rm xt}}^i_{\mathcal{P}}(-,M)=0$ for any $i\in\mathbb{Z}$.
  \item [(4)] If $\xi\mbox{-}{\rm id}N<\infty$, then $\widetilde{\xi{\rm xt}}^i_{\mathcal{P}}(-,N)=0$ for any $i\in\mathbb{Z}$.
\end{itemize}
\end{prop}

\begin{proof}
(1)   (i) $\Rightarrow$ (ii). Since $\xi\mbox{-}\mathcal{G}{\rm pd}M\leq n$, one has a $\xi$-complete resolution $\xymatrix@C=2em{\mathbf{T}\ar[r]^{\nu}&\mathbf{P}\ar[r]^{\pi}&M}$ of $M$ such that $\nu_i$ is bijective for each $i\geq n$. It follows that
$\widetilde{\varepsilon}^i_{\mathcal{P}}(M,N):\xi{\rm xt}^i_{\xi}(M,N)\to \widetilde{\xi{\rm xt}}^i_{\mathcal{P}}(M,N)$ is bijective for any $i\geq n+1$ and any $N\in\mathcal{C}$, as desired.

 (ii) $\Rightarrow$ (i). Assume that $\widetilde{\varepsilon}^i_{\mathcal{P}}(M,N)$ is bijective for any $i\geq n+1$ and any $N\in\mathcal{C}$. In particular, for any $i\geq n+1$ and any $P\in\mathcal{P}(\xi)$, one has $\xi{\rm xt}^i_{\xi}(M,P)\cong \widetilde{\xi{\rm xt}}^i_{\mathcal{P}}(M,P)$.
 But the latter is zero since each complete   $\mathcal{P}(\xi)$-exact resolution is $\mathcal{C}(-,\mathcal{P}(\xi))$-exact. Thus $\xi{\rm xt}^i_{\xi}(M,P)=0$ for any $i\geq n+1$ and any $P\in\mathcal{P}(\xi)$. So $\xi\mbox{-}\mathcal{G}{\rm pd}M\leq n$ by \cite[Theorem 3.8]{HZZ1}.

 (2) Assume that $\xi\mbox{-}{\rm pd}M=n<\infty$. We can choose a $\xi$-projective resolution $\textbf{P}\to M$ of length $n$. Then, in this case,
 $\xymatrix@C=2em{\mathbf{0}\ar[r]^{0}&\mathbf{P}\ar[r]^{\pi}&M}$ is a $\xi$-complete resolution of $M$, and thus  $\widetilde{\xi{\rm xt}}^i_{\mathcal{P}}(M,-)=0$ for any $i\in\mathbb{Z}$.

 (3) and (4) follow from Lemma \ref{3.4} directly.
\end{proof}

Now we show that there is a long exact sequence of $\xi$-complete cohomology under some certain conditions.

\begin{thm}\label{thm3.6}
Let $M\in\widetilde{\mathcal{GP}}(\xi)$ and consider an $\mathbb{E}$-triangle ${\mathfrak{E}}: \xymatrix{A\ar[r]^x& B\ar[r]^y& C \ar@{-->}[r]^{\delta}& }$ in $\xi$.
Then there exist morphisms $\widetilde{\partial}^n(M,{\mathfrak{E}})$, which are natural in $M$ and ${\mathfrak{E}}$, such that the following sequence
$$
\xymatrix@C=1cm{\cdots\ar[r]&\widetilde{\xi{\rm xt}}^n_{\mathcal{P}}(M,A)\ar[r]^{\widetilde{\xi{\rm xt}}^i_{\mathcal{P}}(M,x)}&\widetilde{\xi{\rm xt}}^n_{\mathcal{P}}(M,B)\ar[r]^{\widetilde{\xi{\rm xt}}^i_{\mathcal{P}}(M,y)}& \widetilde{\xi{\rm xt}}^n_{\mathcal{P}}(M,C)\ar[r]^{\widetilde{\partial}^n(M,{\mathfrak{E}})} &\widetilde{\xi{\rm xt}}^{n+1}_{\mathcal{P}}(M,A)\ar[r]&\cdots}
$$
is exact.

Moreover, the connecting map ${\widetilde{\partial}^n(M,{\mathfrak{E}})}$ makes the following diagram
$$
\xymatrix@=1.2cm{\xi{\rm xt}^n_{\xi}(M,C)\ar[r]^{{\partial}^n(M,{\mathfrak{E}})}\ar[d]^{\widetilde{\varepsilon}^n_{\mathcal{P}}(M,C)}&\xi{\rm xt}^{n+1}_{\xi}(M,A)\ar[d]^{\widetilde{\varepsilon}^{n+1}_{\mathcal{P}}(M,A)}\\
\widetilde{\xi{\rm xt}}^n_{\mathcal{P}}(M,C)\ar[r]^{\widetilde{\partial}^n(M,{\mathfrak{E}})}&\widetilde{\xi{\rm xt}}^{n+1}_{\mathcal{P}}(M,A).}
$$
commutative for each $n\in\mathbb{Z}$.
\end{thm}

\begin{proof}
Let $\xymatrix@C=2em{\mathbf{T}\ar[r]^{\nu}&\mathbf{P}\ar[r]^{\pi}&M}$ be a $\xi$-complete resolution of $M$. Then we have a commutative diagram of complexes
$$
\xymatrix{0\ar[r]&\mathcal{C}(\mathbf{P},A)\ar[r]\ar[d]&\mathcal{C}(\mathbf{P},B)\ar[r]\ar[d]&\mathcal{C}(\mathbf{P},C)\ar[r]\ar[d]&0\\
0\ar[r]&\mathcal{C}(\mathbf{T},A)\ar[r]&\mathcal{C}(\mathbf{T},B)\ar[r]&\mathcal{C}(\mathbf{T},C)\ar[r]&0.}
$$
The rows are exact since all terms of $\mathbf{P}$ and $\mathbf{T}$ are $\xi$-projective. By the bottom row we get the desired long exact sequence, and by the commutativity of the diagram and Lemma \ref{lem:complete-cohomology} we get the desired equality. The naturality in $M$ follows from Lemma \ref{com-homotopy},
and the naturality in $\mathfrak{E}$ is clear.
\end{proof}

Using standard arguments
from relative homological algebra, one can prove the following version of the Horseshoe Lemma for $\xi$-complete resolutions.
For convenience, we give the proof.

\begin{lem}\label{com-homotopy2} {\rm (Horseshoe Lemma for $\xi$-complete resolutions)}
Let $\xymatrix{A\ar[r]^x& B\ar[r]^y& C \ar@{-->}[r]^{\delta}& }$ be an $\mathbb{E}$-triangle in $\xi$ such that $\xi\mbox{-}\mathcal{G}{\rm pd}A<\infty$ and $\xi\mbox{-}\mathcal{G}{\rm pd}C<\infty$. Let $\xymatrix@C=2em{\mathbf{T}\ar[r]^{\nu}&\mathbf{P}\ar[r]^{\pi}&A}$ and $\xymatrix@C=2em{\mathbf{T}''\ar[r]^{\nu''}&\mathbf{P}''\ar[r]^{\pi''}&C}$ be  $\xi$-complete resolutions of $A$ and $C$, respectively. Then there is a commutative diagram:
\begin{equation}\label{com}
\begin{split}
\xymatrix{\mathbf{T}\ar[r]\ar[d]^{\nu}&\mathbf{T}'\ar[r]\ar[d]^{\nu'}&\mathbf{T}''\ar[d]^{\nu''}\\
\mathbf{P}\ar[r]\ar[d]^{\pi}&\mathbf{P}'\ar[r]\ar[d]^{\pi'}&\mathbf{P}''\ar[d]^{\pi''}\\
A\ar[r]^x& B\ar[r]^y& C}
\end{split}
\end{equation}
where the two upper rows are split $\mathbb{E}$-triangles in $\xi$ and the columns are $\xi$-complete resolutions.
\end{lem}

\begin{proof}
{By \cite[Lemma 3.3]{HZZ1}, we have a commutative diagram
$$
\xymatrix{
\mathbf{P}\ar[r]\ar[d]^{\pi}&\mathbf{P}'\ar[r]\ar[d]^{\pi'}&\mathbf{P}''\ar[d]^{\pi''}\\
A\ar[r]^x& B\ar[r]^y& C}
$$
where $\mathbf{P}'\to B$ is a $\xi$-projective resolution of $B$. Setting $n=\max\{\xi\mbox{-}\mathcal{G}{\rm pd}A,\xi\mbox{-}\mathcal{G}{\rm pd}C\}$.
Consider an $\mathbb{E}$-triangle $\xymatrix{K_n\ar[r]&K_n'\ar[r]&K_n''\ar@{-->}[r]&,}$ where $K_n$ (respectively, $K_n'$ and $K_n''$) is the $n$th $\xi$-syzygy of $A$ (respectively, $B$ and $C$) obtained from the $\mathbf{P}$ (respectively, $\mathbf{P}'$ and $\mathbf{P}''$).
Then $K_n$ and $K_n''$ are $\xi$-$\mathcal{G}$projective by \cite[Proposition 5.2]{HZZ}. By a suitable adjustment (e.g. see \cite[Proposition 4.4]{HZZZ}), we can require that $\mathbf{T}$ and $\mathbf{T}''$ are complete $\xi$-projective resolutions of $K_n$ and $K_n''$,  respectively. As a similar argument in proof of \cite[Theorem 4.16]{HZZ}, we can construct a complete $\xi$-projective resolution $\mathbf{T}'$ of $K_n'$ such that the desired commutative diagram (\ref{com}) holds.}
\end{proof}

By Lemma \ref{com-homotopy2}, we can get the second long exact sequence of $\xi$-complete cohomology.

\begin{thm}\label{thm:3.11'}
Let ${\mathfrak{E}}: \xymatrix{A\ar[r]^x& B\ar[r]^y& C \ar@{-->}[r]^{\delta}& }$ be an $\mathbb{E}$-triangle in $\xi$ with $\xi\mbox{-}\mathcal{G}{\rm pd}A<\infty$ and $\xi\mbox{-}\mathcal{G}{\rm pd}C<\infty$. For each $N\in\mathcal{C}$,
 there exist morphisms $\widetilde{\partial}^n({\mathfrak{E}},N)$, which are natural in $N$ and ${\mathfrak{E}}$, such that the following sequence
$$
\xymatrix@C=1cm{\cdots\ar[r]&\widetilde{\xi{\rm xt}}^n_{\mathcal{P}}(C,N)\ar[r]^{\widetilde{\xi{\rm xt}}^i_{\mathcal{P}}(y,N)}&\widetilde{\xi{\rm xt}}^n_{\mathcal{P}}(B,N)\ar[r]^{\widetilde{\xi{\rm xt}}^i_{\mathcal{P}}(x,N)}& \widetilde{\xi{\rm xt}}^n_{\mathcal{P}}(A,N)\ar[r]^{\widetilde{\partial}^n({\mathfrak{E}},N)} &\widetilde{\xi{\rm xt}}^{n+1}_{\mathcal{P}}(C,N)\ar[r]&\cdots}
$$
is exact.

Moreover, the connecting map ${\widetilde{\partial}^n({\mathfrak{E}},N)}$  makes the following diagram
$$
\xymatrix@=1.2cm{\xi{\rm xt}^n_{\xi}(A,N)\ar[r]^{{\partial}^n({\mathfrak{E}},N)}\ar[d]^{\widetilde{\varepsilon}^n_{\mathcal{P}}(A,N)}&\xi{\rm xt}^{n+1}_{\xi}(C,N)\ar[d]^{\widetilde{\varepsilon}^{n+1}_{\mathcal{P}}(C,N)}\\
\widetilde{\xi{\rm xt}}^n_{\mathcal{P}}(A,N)\ar[r]^{\widetilde{\partial}^n({\mathfrak{E}},N)}&\widetilde{\xi{\rm xt}}^{n+1}_{\mathcal{P}}(C,N).}
$$
commutative for each $n\in\mathbb{Z}$.
\end{thm}

\begin{proof}
Since $A$ and $C$ have finite $\xi$-$\mathcal{G}$projective dimension, we can construct the diagram (\ref{com}) of $\xi$-complete resolutions.
Moreover, since the two upper rows of (\ref{com}) are split $\mathbb{E}$-triangle in $\xi$, by applying the functor $\mathcal{C}(-,N)$ we can get
a commutative diagram  of complexes
$$
\xymatrix{0\ar[r]&\mathcal{C}(\mathbf{P}'',N)\ar[r]\ar[d]&\mathcal{C}(\mathbf{P}',N)\ar[r]\ar[d]&\mathcal{C}(\mathbf{P},N)\ar[r]\ar[d]&0\\
0\ar[r]&\mathcal{C}(\mathbf{T}'',N)\ar[r]&\mathcal{C}(\mathbf{T}',N)\ar[r]&\mathcal{C}(\mathbf{T},N)\ar[r]&0}
$$
with exact rows.
By the bottom row we get the desired long exact sequence, and by the commutativity of the diagram and Lemma \ref{lem:complete-cohomology} we get the desired equality. The naturality in $N$ follows from Lemma \ref{com-homotopy2},
and the naturality in $\mathfrak{E}$ is clear.
\end{proof}

In the rest of this section, we will give the $\xi$-complete cohomology theory based on $\xi$-$\mathcal{G}$injective objects. All arguments and proofs are similar to the above.

\begin{definition} {\rm Let $N\in\mathcal{C}$ be an object. A \emph{$\xi$-complete coresolution} of $N$ is a diagram $$\xymatrix@C=2em{N\ar[r]^{\iota}&\mathbf{I}\ar[r]^{\mu}&\mathbf{Q}}$$ of morphisms of complexes satisfying:
  (1)  $\iota:N\ra\mathbf{I}$ is a $\xi$-injective coresolution of $N$;
  (2) $\mathbf{Q}$ is a complete $\xi$-injective coresolution;
  (3) $\mu:\mathbf{I}\ra \mathbf{Q}$ is a morphism such that $\mu_{i}$ $=$ {\rm id$_{Q_{i}}$} for all $i\ll 0$.
 }
\end{definition}

An object $N$ in $\mathcal{C}$ admits a $\xi$-complete coresolution if and only if $N$ has finite $\xi$-$\mathcal{G}$injective dimension.

Now assume that $N$ admits a $\xi$-complete coresolution $\xymatrix@C=2em{N\ar[r]^{\iota}&\mathbf{I}\ar[r]^{\mu}&\mathbf{Q}.}$ For each $n\in\mathbb{Z}$ and each $M\in\mathcal{C}$, we have
$$\widetilde{\xi{\rm xt}}_{\mathcal{I}}^n(M,N)\cong H^n(\mathcal{C}(M,\mathbf{Q}))$$

These groups come equipped with comparison morphisms
$$
\widetilde{\varepsilon}^n_{\mathcal{I}}(M,N): {\xi{\rm xt}}_{\xi}^n(M,N) \to \widetilde{\xi{\rm xt}}_{\mathcal{I}}^n(M,N)
$$
given by
$$
H^n\mathcal{C}(M,\mu): H^n\mathcal{C}(M,\mathbf{I})\to H^n\mathcal{C}(M,\mathbf{Q}).
$$

By a dual argument to the above, $\widetilde{\xi{\rm xt}}_{\mathcal{I}}^n$ and $\widetilde{\varepsilon}^n_{\mathcal{I}}$ are independent of the choice of coresolutions and liftings.

\begin{prop}\label{prop3.10}
\begin{itemize}
\item[ ]
  \item [(1)] Let $N\in\widetilde{\mathcal{GI}}(\xi)$. For any integer $n$ the following are equivalent:
  \begin{itemize}
    \item [(i)] $\xi\mbox{-}\mathcal{G}{\rm id}N\leq n$.
    \item [(ii)] The map $\widetilde{\varepsilon}^i_{\mathcal{I}}(M,N):\xi{\rm xt}^i_{\xi}(M,N)\to \widetilde{\xi{\rm xt}}^i_{\mathcal{I}}(M,N)$ is bijective for any $i\geq n+1$ and any $M\in\mathcal{C}$.
  \end{itemize}
  \item [(2)] If $\xi\mbox{-}{\rm id}N<\infty$, then $\widetilde{\xi{\rm xt}}^i_{\mathcal{I}}(N,-)=0$ for any $i\in\mathbb{Z}$.
  \item [(3)] If $\xi\mbox{-}{\rm id}N<\infty$, then $\widetilde{\xi{\rm xt}}^i_{\mathcal{I}}(-,N)=0$ for any $i\in\mathbb{Z}$.
  \item [(4)] If $\xi\mbox{-}{\rm pd}M<\infty$, then $\widetilde{\xi{\rm xt}}^i_{\mathcal{I}}(M,-)=0$ for any $i\in\mathbb{Z}$.
\end{itemize}
\end{prop}

\begin{thm}\label{thm3.11}
\begin{itemize}
\item[ ]
  \item [(1)] Let $N\in\widetilde{\mathcal{GI}}(\xi)$ and consider an $\mathbb{E}$-triangle $\xymatrix{A\ar[r]^x& B\ar[r]^y& C \ar@{-->}[r]^{\delta}& }$ in $\xi$.
Then there exists a long exact sequence
$$
\xymatrix@C=1cm{\cdots\ar[r]&\widetilde{\xi{\rm xt}}^n_{\mathcal{I}}(C,N)\ar[r]^{\widetilde{\xi{\rm xt}}^i_{\mathcal{I}}(y,N)}&\widetilde{\xi{\rm xt}}^n_{\mathcal{I}}(B,N)\ar[r]^{\widetilde{\xi{\rm xt}}^i_{\mathcal{I}}(x,N)}& \widetilde{\xi{\rm xt}}^n_{\mathcal{I}}(A,N)\ar[r] &\widetilde{\xi{\rm xt}}^{n+1}_{\mathcal{I}}(C,N)\ar[r]&\cdots.}
$$
  \item [(2)] Let $\xymatrix{A\ar[r]^x& B\ar[r]^y& C \ar@{-->}[r]^{\delta}& }$ be an $\mathbb{E}$-triangle in $\xi$ with $A,C\in\widetilde{\mathcal{GI}}(\xi)$. For each $M\in\mathcal{C}$,
 there exists a long exact sequence
$$
\xymatrix@C=1cm{\cdots\ar[r]&\widetilde{\xi{\rm xt}}^n_{\mathcal{I}}(M,A)\ar[r]^{\widetilde{\xi{\rm xt}}^i_{\mathcal{I}}(M,x)}&\widetilde{\xi{\rm xt}}^n_{\mathcal{I}}(M,B)\ar[r]^{\widetilde{\xi{\rm xt}}^i_{\mathcal{I}}(M,y)}& \widetilde{\xi{\rm xt}}^n_{\mathcal{I}}(M,C)\ar[r] &\widetilde{\xi{\rm xt}}^{n+1}_{\mathcal{I}}(M,A)\ar[r]&\cdots.}
$$
\end{itemize}
\end{thm}

\section{\bf The balance of $\xi$-complete cohomology}
Our goal in this section is to study the balance of $\xi$-complete cohomology. At first, we recall the following fact taking from \cite{RL}.
\begin{lem}\label{lem4.1} \emph{{\rm (}see \cite[Lemma 3.3]{RL}{\rm)}}
Given a commutative diagram
$$
\xymatrix@=0.7cm{&\vdots\ar[d]&\vdots\ar[d]&\vdots\ar[d]\\
\cdots\ar[r]&M_{2,2}\ar[r]^{d_{2,2}}\ar[d]^{e_{2,2}}&M_{1,2}\ar[r]^{d_{1,2}}\ar[d]^{e_{1,2}}&M_{0,2}\ar[d]^{e_{0,2}}\\
\cdots\ar[r]&M_{2,1}\ar[r]^{d_{2,1}}\ar[d]^{e_{2,1}}&M_{1,1}\ar[r]^{d_{1,1}}\ar[d]^{e_{1,1}}&M_{0,1}\ar[d]^{e_{0,1}}\\
\cdots\ar[r]&M_{2,0}\ar[r]^{d_{2,0}}&M_{1,0}\ar[r]^{d_{1,0}}&M_{0,0}
}
$$
in $\textrm{Ab}$ with all rows and columns exact. Then there are two complexes
$$
\mathbf{C}: \cdots\to {\rm Coker}d_{1,2}\to {\rm Coker}d_{1,1}\to {\rm Coker}d_{1,0}\to 0
$$
and
$$
\mathbf{D}: \cdots\to {\rm Coker}e_{2,1}\to {\rm Coker}e_{1,1}\to {\rm Coker}e_{0,1}\to 0
$$
with $H^n(\mathbf{C})=H^n(\mathbf{D})$ for all $n$.
\end{lem}

The following lemma is the crucial step in the study of the balance of $\xi$-complete cohomology.

\begin{lem}\label{lem4.2}
Assume that $M$ is $\xi$-$\mathcal{G}$projective with a complete $\xi$-projective resolution $\mathbf{T}$, and $N$ is $\xi$-$\mathcal{G}$injective with a complete $\xi$-injective coresolution $\mathbf{Q}$. Then there are isomorphisms
$$
H^i\mathcal{C}(\mathbf{T},N)\cong H^i\mathcal{C}(M,\mathbf{Q})
$$
for all $i\in\mathbb{Z}$.
\end{lem}

\begin{proof}
Consider the following truncated complexes
$$
\mathbf{T}_{\geq 0}: \cdots\to T_2\to T_1\to T_0\to 0
$$
and
$$
\mathbf{Q}_{\leq 0}: 0\to Q_0 \to Q_{-1}\to Q_{-2}\to \cdots.
$$
It is easy to see that $\xymatrix@C=0.5cm{\mathbf{T}\ar[r]&\mathbf{T}_{\geq 0}\ar[r]&M}$ is a $\xi$-complete resolution of $M$, and $\xymatrix@C=0.5cm{N\ar[r]&\mathbf{Q}_{\leq 0}\ar[r]&\mathbf{Q}}$ is a $\xi$-complete coresolution of $N$.  By Propositions \ref{prop3.5} and \ref{prop3.10},
$$
H^i\mathcal{C}(\mathbf{T},N)=\widetilde{\xi{\rm xt}}^i_{\mathcal{P}}(M,N)\cong {\xi{\rm xt}}^i_{\xi}(M,N)\cong\widetilde{\xi{\rm xt}}^i_{\mathcal{I}}(M,N)= H^i\mathcal{C}(M,\mathbf{Q})
$$
for any $i\geq 1$.

Now consider the following $\xi$-exact complexes
$$
\mathbf{T}_{\preceq 0}: 0\to M\to T_{-1}\to T_{-2}\to \cdots
$$
and
$$
\mathbf{Q}_{\succeq 0}: \cdots\to Q_2\to Q_1\to N\to 0.
$$
We have the following commutative diagram
$$
\xymatrix@=0.5cm{
&\vdots\ar[d]&\vdots\ar[d]&\vdots\ar[d]&\\
\cdots\ar[r]&\mathcal{C}(T_{-2},Q_2)\ar[r]\ar[d]&\mathcal{C}(T_{-2},Q_1)\ar[r]\ar[d]&\mathcal{C}(T_{-2},N)\ar[r]\ar[d]&0\\
\cdots\ar[r]&\mathcal{C}(T_{-1},Q_2)\ar[r]\ar[d]&\mathcal{C}(T_{-1},Q_1)\ar[r]\ar[d]&\mathcal{C}(T_{-1},N)\ar[r]\ar[d]&0\\
\cdots\ar[r]&\mathcal{C}(M,Q_2)\ar[r]\ar[d]&\mathcal{C}(M,Q_1)\ar[r]\ar[d]&\mathcal{C}(M,N)\ar[r]\ar[d]&0\\
&0&0&\ 0.&
}
$$
Since $T_{-i}\in\mathcal{P}(\xi)$ and $Q_{i}\in\mathcal{I}(\xi)$ for any $i\geq 1$, all rows and columns are exact except the bottom row $\mathcal{C}(M,\mathbf{Q}_{\succeq 0})$ and the far right column $\mathcal{C}(\mathbf{T}_{\preceq 0},N)$. By Lemma \ref{lem4.1}, the induced complexes
$$
\mathcal{C}(M,\mathbf{Q}_{\geq 1}): \cdots\to\mathcal{C}(M,Q_3)\to\mathcal{C}(M,Q_2)\to\mathcal{C}(M,Q_1)\to 0
$$
and
$$
\mathcal{C}(\mathbf{T}_{\leq -1},N): \cdots\to\mathcal{C}(T_{-3},N)\to\mathcal{C}(T_{-2},N)\to\mathcal{C}(T_{-1},N)\to 0
$$
have isomorphic cohomology groups, that is, for any $i\leq -2$,
$$
H^i\mathcal{C}(\mathbf{T},N)=H^i\mathcal{C}(\mathbf{T}_{\leq -1},N)=H^i\mathcal{C}(M,\mathbf{Q}_{\geq 1})= H^i\mathcal{C}(M,\mathbf{Q}).
$$

Note that for the complete $\xi$-injective coresolution $\mathbf{Q}: \xymatrix@C=0.5cm{\cdots\ar[r]& Q_2\ar[r]^{d_2^{\mathbf{Q}}}& Q_1\ar[r]^{d_1^{\mathbf{Q}}}& Q_0\ar[r]& \cdots}$, there are
$\mathbb{E}$-triangles $\xymatrix@C=0.5cm{L_{i+1}\ar[r]^{s_i}& Q_i\ar[r]^{t_i}& L_i \ar@{-->}[r]^{\delta_i}& }$ (Here $L_1=N$) in $\xi$ such that $d_i^{\mathbf{Q}}=s_{i-1}t_i$ for all $i\in\mathbb{Z}$. Now consider the following commutative diagram
\begin{equation}\label{co}
\begin{split}
\xymatrix@C=0.5cm{
\mathcal{C}(T_{-1},Q_2)\oplus\mathcal{C}(T_{-2},Q_1)\ar[r]^-{\rho_2}\ar[d]_{(0,\mathcal{C}(T_{-2},t_1))}&\mathcal{C}(T_{-1},Q_1)\ar[r]^-{\rho_1}\ar[d]^{\mathcal{C}(T_{-1},t_1)}
&\mathcal{C}(T_0,Q_0)\ar[r]^-{\rho_0}&
\mathcal{C}(T_1,Q_0)\oplus\mathcal{C}(T_0,Q_{-1})\\
\mathcal{C}(T_{-2},N)\ar[r]^{\mathcal{C}(d_{-1}^{\mathbf{T}},N)}&\mathcal{C}(T_{-1},N)\ar[r]^{\mathcal{C}(d_{0}^{\mathbf{T}},N)}&
\mathcal{C}(T_0,N)\ar[u]_{\mathcal{C}(T_0,s_0)}\ar[r]^{\mathcal{C}(d_{1}^{\mathbf{T}},N)}&\mathcal{C}(T_1,N)\ar[u]_{\mathcal{C}(T_1,s_0)\choose 0}
}
\end{split}
\end{equation}
where
\begin{align*}
  \rho_2: &\ \ \mathcal{C}(T_{-1},Q_2)\oplus\mathcal{C}(T_{-2},Q_1)\to\mathcal{C}(T_{-1},Q_1)\\
   & \ \ \ \ \ \ \ \ \ \ \  \alpha:=(\alpha_{-1},\alpha_{-2}) \  \mapsto \ d_2^{\mathbf{Q}}\alpha_{-1}+\alpha_{-2}d_{-1}^{\mathbf{T}}\\
  \rho_1: &\ \ \mathcal{C}(T_{-1},Q_1)\to \mathcal{C}(T_0,Q_0)\\
  & \ \ \ \ \ \ \ \ \ \ \ \
   \beta \ \mapsto \ \ d_1^{\mathbf{Q}}\beta d_{0}^{\mathbf{T}}\\
 \rho_0: & \ \ \mathcal{C}(T_0,Q_0)\to \mathcal{C}(T_1,Q_0)\oplus\mathcal{C}(T_0,Q_{-1}). \\
   & \ \ \ \ \ \ \ \ \ \ \  \gamma \ \mapsto \ (\gamma d_{1}^{\mathbf{T}}, d_0^{\mathbf{Q}}\gamma)
\end{align*}
Define
\begin{align*}
  \Phi: & \ \ {\rm Ker}\rho_1/{\rm Im} \rho_2 \to {\rm Ker}\mathcal{C}(d_0^{\mathbf{T}},N)/{\rm Im}\mathcal{C}(d_{-1}^{\mathbf{T}},N).  \\
   & \ \ \ \ \ \alpha+{\rm Im}\rho_2 \ \mapsto \ t_1\alpha+{\rm Im}\mathcal{C}(d_{-1}^{\mathbf{T}},N)
\end{align*}
We first show that it is a well-defined map. Indeed, let $\alpha\in{\rm Ker}\rho_1$. Then
$$0=\rho_1(\alpha)=d_1^{\mathbf{Q}}\alpha d_{0}^{\mathbf{T}}=s_0t_1\alpha d_{0}^{\mathbf{T}}=\mathcal{C}(T_0,s_0)(t_1\alpha d_{0}^{\mathbf{T}}).$$
Since $T_0\in\mathcal{P}(\xi)$, $\mathcal{C}(T_0,s_0)$ is injective, and hence $\mathcal{C}(d_{0}^{\mathbf{T}},N)(t_1\alpha)=t_1\alpha d_{0}^{\mathbf{T}}=0$. Thus $t_1\alpha\in{\rm Ker}\mathcal{C}(d_{0}^{\mathbf{T}},N)$. Moreover, if $\beta\in {\rm Im} \rho_2$, then $t_1 \beta\in{\rm Im}\mathcal{C}(d_{-1}^{\mathbf{T}},N)$ by the commutativity of left square of (\ref{co}). This shows that $\Phi$ is a well-defined map.

We next show that $\Phi$ is an isomorphism. Firstly, since $T_{-1}\in\mathcal{P}(\xi)$, there is an exact sequence
 $$
 \xymatrix{0\ar[r]&\mathcal{C}(T_{-1},L_2)\ar[r]^{\mathcal{C}(T_{-1},s_1)}&\mathcal{C}(T_{-1},Q_1)\ar[r]^{\mathcal{C}(T_{-1},t_1)}&
 \mathcal{C}(T_{-1},N)\ar[r]&0.
}
 $$
In particular, the induced map $\mathcal{C}(T_{-1},t_1)$ is surjective, and following this we easily get  that $\Phi$ is surjective. Now assume that $\Phi(\alpha+{\rm Im}\rho_2)=0$, that is, $t_1\alpha\in{\rm Im}\mathcal{C}(d_{-1}^{\mathbf{T}},N)$. Then there is $\beta\in\mathcal{C}(T_{-2},N)$ with $t_1\alpha=\mathcal{C}(d_{-1}^{\mathbf{T}},N)(\beta)$. Similarly, since $T_{-2}\in\mathcal{P}(\xi)$, the induced map $\mathcal{C}(T_{-2},t_1)$ is surjective, and hence there is $\gamma:=(\gamma_{-1},\gamma_{-2})\in\mathcal{C}(T_{-1},Q_2)\oplus\mathcal{C}(T_{-2},Q_1)$ with $(0,\mathcal{C}(T_{-2},t_1))(\gamma)=\beta$.
By the commutativity of left square of (\ref{co}), $\mathcal{C}(T_{-1},t_1)(\alpha)=t_1\alpha=\mathcal{C}(d_{-1}^{\mathbf{T}},N)(\beta)=
\mathcal{C}(d_{-1}^{\mathbf{T}},N)((0,\mathcal{C}(T_{-2},t_1))(\gamma))=\mathcal{C}(T_{-1},t_1)(\rho_2(\gamma))$,
which shows that $\mathcal{C}(T_{-1},t_1)(\alpha-\rho_2(\gamma))=0$, that is, $\alpha-\rho_2(\gamma)\in{\rm Ker}\mathcal{C}(T_{-1},t_1)={\rm Im}\mathcal{C}(T_{-1},s_1)$. Thus there exists $\delta\in \mathcal{C}(T_{-1},L_2)$ such that $\alpha-\rho_2(\gamma)=\mathcal{C}(T_{-1},s_1)(\delta)$. By the surjectivity of $\mathcal{C}(T_{-1},t_2)$, there is $\epsilon\in\mathcal{C}(T_{-1},Q_2)$ with $\delta=\mathcal{C}(T_{-1},t_2)(\epsilon)$. It follows that
\begin{align*}
  \alpha & =\rho_2(\gamma)+s_1\delta= d_2^{\mathbf{Q}}\gamma_{-1}+\gamma_{-2}d_{-1}^{\mathbf{T}}+s_1 t_1 \epsilon \\
   & =d_2^{\mathbf{Q}}\gamma_{-1}+\gamma_{-2}d_{-1}^{\mathbf{T}}+d_2^{\mathbf{Q}}\epsilon=d_2^{\mathbf{Q}}(\gamma_{-1}+\epsilon)+\gamma_{-2}d_{-1}^{\mathbf{T}}\\
&=\rho_2(\gamma_{-1}+\epsilon,\gamma_{-2})
\end{align*}
which means that $\alpha\in{\rm Im}\rho_2$. This shows that $\Phi$ is injective, and hence $\Phi$ is an isomorphism.

Similarly, we can show that the following map
\begin{align*}
  \Psi: & \ \ {\rm Ker}\mathcal{C}(d_1^{\mathbf{T}},N)/{\rm Im}\mathcal{C}(d_0^{\mathbf{T}},N)\to {\rm Ker}\rho_0/{\rm Im}\rho_1 \\
   & \ \ \ \ \ \ \ \ \ \  \alpha+{\rm Im}\mathcal{C}(d_0^{\mathbf{T}},N) \ \mapsto \ s_0\alpha+{\rm Im}\rho_1
\end{align*}
is an isomorphism.

On the other hand, we can prove that there is a commutative diagram as follows
\begin{equation}\label{co2}
\begin{split}
\xymatrix@C=0.5cm{
\mathcal{C}(T_{-1},Q_2)\oplus\mathcal{C}(T_{-2},Q_1)\ar[r]^-{\rho_2}\ar[d]&\mathcal{C}(T_{-1},Q_1)\ar[r]^-{\rho_1}\ar[d]
&\mathcal{C}(T_0,Q_0)\ar[r]^-{\rho_0}&
\mathcal{C}(T_1,Q_0)\oplus\mathcal{C}(T_0,Q_{-1})\\
\mathcal{C}(M,Q_2)\ar[r]&\mathcal{C}(M,Q_1)\ar[r]&
\mathcal{C}(M,Q_0)\ar[u]\ar[r]&\mathcal{C}(M,Q_{-1})\ar[u]
}
\end{split}
\end{equation}
and the two rows have isomorphic cohomological groups. It follows that the bottom row of (\ref{co}) and the bottom row of (\ref{co2}) have isomorphic cohomological groups, that is,
$H^i\mathcal{C}(\mathbf{T},N)\cong H^i\mathcal{C}(M,\mathbf{Q})
$ for $i=0,-1$.
\end{proof}

More generally, we have the balance of $\xi$-complete cohomology as follows, which is the key result to prove the main result, and meanwhile generalizes \cite[Theorem 4.11]{AS2}, \cite[Theorem 2]{Ia} and \cite[Main Theorem]{RL}.

\begin{prop}\label{prop:key-result}
Let $M\in\widetilde{\mathcal{GP}}(\xi)$ and $N\in\widetilde{\mathcal{GI}}(\xi)$. Then
$$
\widetilde{\xi{\rm xt}}^i_{\mathcal{P}}(M,N)\cong \widetilde{\xi{\rm xt}}^i_{\mathcal{I}}(M,N)
$$
for any $i\in\mathbb{Z}$.
\end{prop}

\begin{proof}
Assume that $\xi$-$\mathcal{G}$pd$M=m<\infty$ and $\xi$-$\mathcal{G}$id$N=n<\infty$. Let $\xymatrix@C=2em{\mathbf{T}\ar[r]^{\nu}&\mathbf{P}\ar[r]^{\pi}&M}$ be a $\xi$-complete resolution of $M$ and $\xymatrix@C=2em{N\ar[r]^{\iota}&\mathbf{I}\ar[r]^{\mu}&\mathbf{Q}}$ a $\xi$-complete coresolution of $N$.

For any $i\in\mathbb{Z}$, there are $\mathbb{E}$-triangles
$$
\xymatrix@C=0.5cm{K_{i+1}\ar[r]&T_i\ar[r]&K_i\ar@{-->}[r]&,} \ \xymatrix@C=0.5cm{D_{i+1}\ar[r]&P_i\ar[r]&D_i\ar@{-->}[r]&}
$$
and
$$
\xymatrix@C=0.5cm{J_{i+1}\ar[r]&I_i\ar[r]&J_i\ar@{-->}[r]&,} \ \xymatrix@C=0.5cm{L_{i+1}\ar[r]&Q_i\ar[r]&L_i\ar@{-->}[r]&}
$$
in $\xi$. Set $M=D_0$ and $N=J_1$. Then $K_m=D_m$ is $\xi$-$\mathcal{G}$projective and $L_{-n+1}=J_{-n+1}$ is $\xi$-$\mathcal{G}$injective.

Consider the truncations
$$
\mathbf{T}_{\geq m}: \ \cdots\to T_{m+2}\to T_{m+1}\to T_m \to 0
$$
and
$$
\mathbf{Q}_{\leq -n}: \ 0\to  Q_{-n}\to Q_{-n-1}\to Q_{-n-2}\to \cdots.
$$
Then $\mathbf{T}[-m]\to \mathbf{T}_{\geq m}[-m]\to K_m$ is a $\xi$-complete resolution of $K_m$ and $L_{-n+1}\to \mathbf{Q}_{\leq -n}[n]\to \mathbf{Q}[n]$ is
 a $\xi$-complete coresolution of $L_{-n+1}$. Thus
  we have
 \begin{align*}
   \widetilde{\xi{\rm xt}}^i_{\mathcal{P}}(M,N)&=H^i\mathcal{C}(\mathbf{T},N)
   =H^{i-m}\mathcal{C}(\mathbf{T}[-m],N)
   =\widetilde{\xi{\rm xt}}^{i-m}_{\mathcal{P}}(K_m,N)\\
   &\cong \widetilde{\xi{\rm xt}}^{i-m-n}_{\mathcal{P}}(K_m,J_{-n+1}) \ \ (\mbox{by Proposition \ref{prop3.5}(4) and Theorem \ref{thm3.6} })\\
  & =\widetilde{\xi{\rm xt}}^{i-m-n}_{\mathcal{P}}(K_m,L_{-n+1})
    =H^{i-m-n}\mathcal{C}(\mathbf{T}[-m],L_{-n+1}) \\
    &\cong H^{i-m-n}\mathcal{C}(K_m,\mathbf{Q}[n])\ \ ({{\rm by \  Lemma} \ \ref{lem4.2}})\\
    &=H^{i-m}\mathcal{C}(K_m,\mathbf{Q})\\
    &=\widetilde{\xi{\rm xt}}^i_{\mathcal{I}}(K_m,N).
 \end{align*}
 Dually, using Proposition \ref{prop3.10} and Theorem \ref{thm3.11} we have
 $$
 \widetilde{\xi{\rm xt}}^i_{\mathcal{I}}(M,N)\cong \widetilde{\xi{\rm xt}}^{i-m}_{\mathcal{I}}(D_m,N)=\widetilde{\xi{\rm xt}}^{i-m}_{\mathcal{I}}(K_m,N).
 $$
 Therefore, $\widetilde{\xi{\rm xt}}^i_{\mathcal{P}}(M,N)\cong \widetilde{\xi{\rm xt}}^i_{\mathcal{I}}(M,N)$ for any $i\in\mathbb{Z}$.
\end{proof}

Motivated by Gedrich and Gruenberg's invariants of a ring \cite{GG}, and Asadollahi and Salarian's invariants to a triangulated category \cite{AS2}, the authors  assigned in \cite[Definition 4.2]{HZZ1} two invariants to an extriangulated category $\mathcal{C}$:
 $$\xi\textrm{-}{\rm silp}\mathcal{C}=\sup\{\xi\textrm{-}{\rm id}P \ | \ P\in{\mathcal{P}(\xi)}\},$$
$$\xi\textrm{-}{\rm spli}\mathcal{C}=\sup\{\xi\textrm{-}{\rm pd}I \ | \ I\in{\mathcal{I}(\xi)}\}.$$
It follows from \cite[Proposition 4.3]{HZZ1} that if both $\xi\textrm{-}{\rm silp}\mathcal{C}$ and $\xi\textrm{-}{\rm spli}\mathcal{C}$ are finite, then they are equal.

Recall from \cite[Definition 4.1]{HZZ1} that a full subcategory $\mathcal{X}\subseteq \mathcal{C}$ is called a {\it generating subcategory} of $\mathcal{C}$ if for all $M\in{\mathcal{C}}$, $\mathcal{C}(\mathcal{X},M)=0$ implies that $M=0$. Dually, a full subcategory $\mathcal{Y}\subseteq\mathcal{C}$ is called a \emph{cogenerating subcategory} of $\mathcal{C}$ if for
all $N\in{\mathcal{C}}$, $\mathcal{C}(N,\mathcal{Y})=0$ implies that $N=0$.

Next, we are in a position to prove the main result of this section.

\begin{thm}\label{thm:3.4}Let $\mathcal{C}$ be an extriangulated category, and let $\mathcal{P}(\xi)$ be a generating subcategory of $\mathcal{C}$ and $\mathcal{I}(\xi)$ a cogenerating subcategory of $\mathcal{C}$. If $\widetilde{\xi{\rm xt}}^i_{\mathcal{P}}(M,N)\cong \widetilde{\xi{\rm xt}}^i_{\mathcal{I}}(M,N)$ for all objects $M$ and $N$ in $\mathcal{C}$, then $\xi\textrm{-}{\rm spli}\mathcal{C}=\xi\textrm{-}{\rm silp}\mathcal{C}<\infty$.
 The converses hold if $\mathcal{C}$ satisfies the following condition:

{\rm \textbf{Condition $(\star)$}}: If  $N\in\mathcal{C}$ and $M\in{\widetilde{\mathcal{P}}(\xi)}$  such that $\xi{\rm xt}_{\xi}^{i}(M,N)=0$ for  any $i\geq1$, then $\mathcal{C}(M,N)\cong\xi{\rm xt}_{\xi}^{0}(M,N)$. Dually, if $N\in\mathcal{C}$  and $M\in{\widetilde{\mathcal{I}}(\xi)}$  such that $\xi{\rm xt}_{\xi}^{i}(N,M)=0$ for any  $i\geq1$, then $\mathcal{C}(N,M)\cong\xi{\rm xt}_{\xi}^{0}(N,M)$.
\end{thm}
\begin{proof} Assume that $\widetilde{\xi{\rm xt}}^i_{\mathcal{P}}(M,N)\cong \widetilde{\xi{\rm xt}}^i_{\mathcal{I}}(M,N)$ for all objects $M$ and $N$ in $\mathcal{C}$. Let $M$ be an object in $\mathcal{P}(\xi)$. Then $\widetilde{\xi{\rm xt}}^i_{\mathcal{P}}(M,M)=0$. Thus $\widetilde{\xi{\rm xt}}^i_{\mathcal{I}}(M,M)=0$, and hence $\xi\textrm{-}{\rm id}M<\infty$ by \cite[Theorem 3.10]{HZZZ}. Similarly, one can show that $\xi\textrm{-}{\rm pd}N<\infty$ for any object in $\mathcal{I}(\xi)$. So $\xi\textrm{-}{\rm spli}\mathcal{C}=\xi\textrm{-}{\rm silp}\mathcal{C}<\infty$ by \cite[Proposition 4.3]{HZZ1}.

By hypothesis and \cite[Theorem 4.7]{HZZ1}, we get that every object in $\mathcal{C}$ has finite $\xi$-$\mathcal{G}$projective dimension and finite $\xi$-$\mathcal{G}$injective dimension. So $\widetilde{\xi{\rm xt}}^i_{\mathcal{P}}(M,N)\cong \widetilde{\xi{\rm xt}}^i_{\mathcal{I}}(M,N)$ for all objects $M$ and $N$ in $\mathcal{C}$ by Proposition \ref{prop:key-result}.
\end{proof}

\begin{Ex} {\emph{(see \cite[Example 4.10]{HZZ1})}} \label{Ex:3.12}
\emph{(1)} Assume that $(\mathcal{C}, \mathbb{E}, \mathfrak{s})$ is an exact category and $\xi$ is a class of exact sequences which is closed under isomorphisms. One can check that Condition $(\star)$ in Theorem \ref{thm:3.4} is automatically satisfied.

\emph{(2)} If $\mathcal{C}$ is a triangulated category and the class $\xi$ of triangles is closed under
isomorphisms and suspension \emph{(see \cite[Section 2.2]{Bel1} and \cite[Remark 3.3(3)]{HZZ})}, then Condition $(\star)$ in Theorem \ref{thm:3.4} is also satisfied.
\end{Ex}

Recall from \cite[Definition 4.6]{AS2} that a triangulated category $\mathcal{C}$ is called \emph{Gorenstein} if there exists a
positive integer $n$ such that any object of $\mathcal{C}$ has both $\xi$-$\mathcal{G}$projective and $\xi$-$\mathcal{G}$injective dimension less than or equal to $n$.

\begin{cor} \label{cor:3.7}Let $\mathcal{C}$ be a triangulated category, and let $\mathcal{P}(\xi)$ be a generating subcategory of $\mathcal{C}$ and $\mathcal{I}(\xi)$ a cogenerating subcategory of $\mathcal{C}$. Then $\mathcal{C}$ is Gorenstein if and only if $\widetilde{\xi{\rm xt}}^i_{\mathcal{P}}(M,N)\cong \widetilde{\xi{\rm xt}}^i_{\mathcal{I}}(M,N)$ for all objects $M$ and $N$ in $\mathcal{C}$.
\end{cor}
\begin{proof} The result holds by Theorem \ref{thm:3.4}, Example \ref{Ex:3.12}(2) and \cite[Corollary 4.12]{HZZ1}.
\end{proof}

Let $R$ be a ring and Mod$R$ the category of left $R$-modules. Then it is clear that the class of projective left $R$-modules is a generating subcategory of Mod$R$ and the class of injective $R$-modules is a cogenerating subcategory of Mod$R$. If we assume that $\mathcal{C}$ is the category of left $R$-modules and $\xi$ is the class of all exact sequences in Mod$R$, then by Theorem \ref{thm:3.4} and Example \ref{Ex:3.12}(1), we have the following corollary, which has been proved by Nucinkis in  \cite{Nucinkis}.

\begin{cor}{\rm (}see \cite[Theorem 5.2]{Nucinkis}{\rm)} Let $R$ be a ring. Then $\widetilde{\xi{\rm xt}}^i_{\mathcal{P}}(M,N)\cong \widetilde{\xi{\rm xt}}^i_{\mathcal{I}}(M,N)$ for all left $R$-modules $M$ and $N$ if and only if  ${\rm spli}R={\rm silp}R<\infty$.
\end{cor}

Let $\Lambda$ be an artin algebra over a commutative artinian ring $k$ and $\textrm{mod}\Lambda$
denotes the category of finitely generated left $\Lambda$-modules. We denote by $(\textrm{mod}\Lambda)^{op}$
the opposite category of $\textrm{mod}\Lambda$. Suppose $F$ is an additive subbifunctor of the additive bifunctor $\textrm{Ext}_{\Lambda}^{1}(-,-):(\textrm{mod}\Lambda)^{op}\times\textrm{mod}\Lambda\rightarrow \textrm{Ab}$. Recall from \cite{Auslander} that a short exact sequence $\eta: 0\rightarrow A\rightarrow B\rightarrow C\rightarrow 0$ in $\textrm{mod}\Lambda$ is said to be $F$-exact if $\eta$ is in $F(C,A)$. A $\Lambda$-module $P$ (resp. $I$) in $\textrm{mod}\Lambda$ is said to be $F$-projective
(resp. $F$-injective) if for each $F$-exact sequence $0\rightarrow A\rightarrow B\rightarrow C\rightarrow 0$,
the sequence $0\rightarrow \textrm{Hom}_{\Lambda}(P,A) \rightarrow \textrm{Hom}_{\Lambda}(P,B)\rightarrow \textrm{Hom}_{\Lambda}(P,C)\rightarrow 0$ (resp.
$0\rightarrow \textrm{Hom}_{\Lambda}(C,I) \rightarrow \textrm{Hom}_{\Lambda}(B,I)\rightarrow \textrm{Hom}_{\Lambda}(A,I)\rightarrow 0$) is exact. The full subcategory
of $\textrm{mod}\Lambda$ consisting of all $F$-projective (resp. $F$-injective) modules is denoted
by $\mathcal{P}(F)$ (resp. $\mathcal{I}(F)$).
Assume that $F$ has enough projectives and injectives. Then $(\textrm{mod}\Lambda,\varepsilon)$ is an exact category, where  $\varepsilon$ is the class of $F$-exact sequences.


On the other hand,
it follows from \cite[Theorem 3.4]{tang} that $\Lambda$ is $F$-Gorenstein if and only if $\sup\{\textrm{pd}_{F} I \ | \ I \in{\mathcal{I}(F)}\} = \sup\{\textrm{id}_{F}P \ | \ P\in{\mathcal{P}(F)}\}<\infty$, where $\textrm{pd}_{F}I =\textrm{inf}\{n \ | \ \textrm{Ext}_{F}^{n+1}(I,B)=0 \ \textrm{for} \ \textrm{any} \ B \in{\textrm{mod}\Lambda}\}$ and $\textrm{id}_{F}P =\textrm{inf}\{n \ | \ \textrm{Ext}_{F}^{n+1}(A,P)=0 \ \textrm{for} \ \textrm{any} \ A \in{\textrm{mod}\Lambda}\}$.  By Theorem \ref{thm:3.4} and Example \ref{Ex:3.12}(1), we have the following corollary which characterizes when an arrtin algebra is $F$-Gorenstein.

\begin{cor}\label{cor3} Let $\Lambda$ be an artian algebra and $F$ an additive subbifunctor with enough projectives and injectives.  If we assume that $\mathcal{C}$ is the exact category {\rm $(\textrm{mod}\Lambda,\varepsilon)$} and $\xi=\varepsilon$ is the class of $F$-exact sequences in {\rm$\textrm{mod}\Lambda$}, then $\Lambda$ is $F$-Gorenstein if and only if $\widetilde{\xi{\rm xt}}^i_{\mathcal{P}}(M,N)\cong \widetilde{\xi{\rm xt}}^i_{\mathcal{I}}(M,N)$ for all finitely generated left $\Lambda$-modules $M$ and $N$.
\end{cor}

\bigskip

\renewcommand\refname
\vspace{4mm}
\small

\hspace{-1.2em}\textbf{Jiangsheng Hu}\\
School of Mathematics and Physics, Jiangsu University of Technology,
 Changzhou 213001, China\\
E-mail: jiangshenghu@jsut.edu.cn\\[1mm]
\textbf{Dongdong Zhang}\\
Department of Mathematics, Zhejiang Normal University,
 Jinhua 321004, China\\
E-mail: zdd@zjnu.cn\\[1mm]
\textbf{Tiwei Zhao}\\
School of Mathematical Sciences, Qufu Normal University, Qufu 273165, China\\
E-mail: tiweizhao@qfnu.edu.cn\\[1mm]
\textbf{Panyue Zhou}\\
College of Mathematics, Hunan Institute of Science and Technology, Yueyang 414006, China\\
E-mail: panyuezhou@163.com


{\bf References}
\begin{thebibliography}{99}


\bibitem{AS2} J. Asadollahi and S. Salarian, \emph{Tate cohomology and Gorensteinness for triangulated categories}, J. Algebra 299 (2006) 480--502.

\bibitem{Auslander} M. Auslander and $\varnothing$. Solberg, \emph{Relative homology and representation theory I. Relative
homology and homologically finite subcategories}, Comm. Algebra 21 (1993) 2995--3031.
\bibitem{AM} L.L. Avramov and A. Martsinkovsky, \emph{Absolute, relative, and Tate cohomology of modules of finite Gorenstein dimension}, Proc. London Math. Soc. 85(3) (2002) 393--440.
  \bibitem{AFH} L.L. Avramov, H.-B. Foxby and S. Halperin, \emph{Differential graded homological algebra}, preprint, 2009.
\bibitem{Bel1} A. Beligiannis, \emph{Relative homological algebra and purity in  triangulated categories}, J. Algebra 227(1) (2000) 268--361.

\bibitem{BF} D.J. Benson and J.F. Carlson, \emph{Products in negative cohomology}, J. Pure Appl. Algebra 82 (1992) 107--130.

\bibitem{B"u} T. B\"{u}hler, \emph{Exact categories}, Expo. Math. 28 (2010) 1--69.

%
%
%
%
%

\bibitem{CFH} L.W. Christensen,  H.-B. Foxby and H. Holm, \emph{Derived Category Methods in Commutative Algebra},
 preprint, 2019.

\bibitem{GG} T.V.  Gedrich and K.W.  Gruenberg, \emph{Complete cohomological functors on groups}, {Topol. Appl.} 25 (1987) 203--223.

\bibitem{Goichot} F. Goichot, \emph{Homologie de Tate-Vogel $\acute{e}$quivariante}, J. Pure Appl. Algebra 82 (1992) 39--64.

\bibitem{HZZ} J.S. Hu, D.D. Zhang and P.Y. Zhou, \emph{Proper classes and Gorensteinness in extriangulated categories},
J. Algebra 551 (2020) 23--60.
\bibitem{HZZ1} J.S. Hu, D.D. Zhang and P.Y. Zhou, \emph{Gorenstein homological dimensions for extriangulated categories},
arXiv:1908.00931
\bibitem{HZZZ} J.S. Hu, D.D. Zhang, T.W. Zhao and P.Y. Zhou, \emph{Complete cohomology for extriangulated categories}, arXiv:2003.11852v2.

\bibitem{Ia} A. Iacob, \emph{Generalized Tate cohomology}, Tsukuba J. Math. 29 (2005) 389--404.
\bibitem{LN} Y. Liu and H. Nakaoka, \emph{Hearts of twin cotorsion pairs on extriangulated categories}, J. Algebra 528 (2019) 96--149.
%

\bibitem{Mislin} G. Mislin, \emph{Tate cohomology for arbitrary groups via satellites}, Topology Appl. 56 (1994) 293--300.
\bibitem{NP}  H. Nakaoka and Y. Palu,  \emph{Extriangulated categories, Hovey twin cotorsion pairs and model structures}, Cah. Topol. G\'{e}om. Diff\'{e}r. Cat\'{e}g.  60(2) (2019) 117--193.

\bibitem{Nucinkis} B.E.A. Nucinkis, \emph{Complete cohomology for arbitrary rings using injectives}, J. Pure Appl. Algebra 131(3) (1998) 297--318.
\bibitem{RL} W. Ren and Z.K. Liu, \emph{Balance of Tate cohomology  in triangulated categories}, Appl. Categ.  Struct. 23(6) (2015) 819--828.
\bibitem{RL3} W. Ren, R.Y. Zhao and Z.K. Liu, \emph{Cohomology theoreies in triangulated categories}, Acta Math. Sin. Engl. Ser. 32(11) (2016) 1377--1390.

\bibitem{tang} X. Tang, \emph{On F-Gorenstein dimensions}, J. Algebra Appl. 13(6) (2014) 1450022 (14 pages).
\bibitem{YC} X.Y. Yang and W.J. Chen,  \emph{Relative homological dimensions and Tate cohomology of complexes with respect to cotorsion pairs}, Comm. Algebra 45(7) (2017) 2875--2888.
%
%
%
 \bibitem{ZAD} F. Zareh-Khoshchehreh, M. Asgharzadeh, and K. Divaani-Aazar, \emph{Gorenstein homology, relative pure homology
and virtually Gorenstein rings}, J. Pure Appl. Algebra 218(12) (2018) 2356--2366.
 \bibitem{ZZ} P.Y. Zhou and B. Zhu,  \emph{Triangulated quotient categories revisited}, J. Algebra 502 (2018) 196--232.

\end{thebibliography}
\end{document}